\documentclass[11pt]{amsart}
\usepackage{amssymb,pb-diagram}
\usepackage{amsmath}
\usepackage{amsthm,amssymb,amsfonts}
\usepackage{amssymb,amscd}
\usepackage{epic,eepic,epsfig}
\usepackage{latexsym}
\usepackage{xy}
\usepackage[usenames, dvipsnames]{color}

\usepackage[?]{amsrefs}

\xyoption{all}

\newtheorem{Thm}{Theorem}[section]

\newtheorem{Cor}[Thm]{Corollary}
\newtheorem{Lem}[Thm]{Lemma}
\newtheorem{Prop}[Thm]{Proposition}

\theoremstyle{definition}
\newtheorem{Def}[Thm]{Definition}

\theoremstyle{remark}

\def \cal{\mathcal}

\def\eps{\varepsilon}

\def\Ndb{\mathbb N}

\begin{document}

\title{Some properties of coarse Lipschitz maps between Banach spaces}

\author{A. Dalet}

\address{Univ. Bourgogne Franche-Comt\'{e}, Laboratoire de Math\'{e}matiques de Besan\c{c}on UMR 6623,
16 route de Gray, 25030 Besan\c{c}on Cedex, FRANCE.}
\email{aude.dalet@univ-fcomte.fr}

\author{G. Lancien}

\address{Univ. Bourgogne Franche-Comt\'{e}, Laboratoire de Math\'{e}matiques de Besan\c{c}on UMR 6623,
16 route de Gray, 25030 Besan\c{c}on Cedex, FRANCE.}
\email{gilles.lancien@univ-fcomte.fr}


\subjclass[2010]{Primary 46B80; Secondary 46B03, 46B20}
\thanks{}

\keywords{coarse Lipschitz maps and equivalences, asymptotically uniformly smooth norms}

\maketitle

\begin{abstract} We study the structure of the space of coarse Lipschitz maps between Banach spaces. In particular we introduce the notion of norm attaining coarse Lipschitz maps. We extend to the case of norm attaining coarse Lipschitz equivalences, a result of G. Godefroy on Lipschitz equivalences. This leads us to include the non separable versions of classical results on the stability of the existence of asymptotically uniformly smooth norms under Lipschitz or coarse Lipschitz equivalences.
\end{abstract}

\section{Introduction} In a recent paper \cite{Godefroy2016} G. Godefroy studied various notions of norm attaining Lipschitz functions. If $(M,d)$ and $(N,\delta)$ are two metric spaces and $f:M\to N$ is Lipschitz, it is natural to say that $f$ attains its norm at the pair $(x,y)$ in $M\times M$ with $x\neq y$ if
$$\frac{\delta(f(x),f(y))}{d(x,y)}=Lip\,(f),$$
where $Lip\,(f)$ denotes the Lipschitz constant of $f$.\\
In \cite{Godefroy2016}, G. Godefroy introduced the following weaker form of norm attaining vector valued Lipschitz functions. Let $(M,d)$ be a metric space, $(Y,\|\ \|_Y)$ a Banach space and $f:M\to Y$ a Lipschitz map. We say that $f$ attains its norm in the direction $y\in S_Y$, where $S_Y$ denotes the unit sphere of $Y$, if there exists a sequence $(s_n,t_n)_{n=0}^\infty$ in $M\times M$ with $s_n\neq t_n$ and such that
$$\lim_{n\to \infty}\frac{f(s_n)-f(t_n)}{d(s_n,t_n)}=y\,Lip\,(f).$$
One of the main results of \cite{Godefroy2016} is that if a Lipschitz isomorphism $f$ between two Banach spaces $X$ and $Y$ attains its norm in the direction $y\in S_Y$, then there exists a constant $c>0$ such that $\overline{\rho}_Y(y,ct)\le 2\overline{\rho}_X(t)$, where $\overline{\rho}$ denotes the modulus of asymptotic uniform smoothness (see definitions in section \ref{auscle}). Then, noticing that this is impossible if one of the spaces is asymptotically uniformly flat and the other has a norm with the Kadets-Klee property, he provides examples of pairs of Banach spaces $(X,Y)$ for which the set of norm attaining Lipschitz maps, in this weaker sense, is not dense.

\medskip The starting point of this work was to notice that this argument could be adapted to the setting of coarse Lipschitz maps between Banach spaces $X$ and $Y$. This space of functions is a vector space on which a natural semi-norm is given by the Lipschitz constant at infinity of a coarse Lipschitz map. It is then natural to work with the corresponding quotient space that we shall denote $CL(X,Y)$.

\smallskip
In section \ref{clm} we introduce these basic definitions as well as the analogue of Godefroy's definition for norm attaining coarse Lipschitz maps.

\smallskip
In section \ref{cle} we define the notion of coarse Lipschitz equivalent Banach spaces, or quasi-isometric Banach spaces in the terminology introduced by M. Gromov in \cite{Gromov1987}. In Proposition \ref{CLE}, we gather some characterizations of the coarse Lipschitz equivalence between Banach spaces that were essentially known. In particular we describe the link with the notion of net equivalence of Banach spaces. We also insist on the existence of continuous representatives of coarse Lipschitz equivalences. This will be crucial in our further use of the Gorelik principle.

\smallskip
In section \ref{completeness} we address the question of the completeness of our normed quotient space $CL(X,Y)$. In Proposition \ref{complete} we give a sufficient condition for $CL(X,Y)$ to be complete. We also describe situations when the coarse Lipschitz equivalences can be viewed as an open subset of our quotient space.

\smallskip
In section \ref{Gorelik} we gather the necessary background on the so-called Gorelik principle. First, we recall its classical version for uniform homeomorphisms and Lipschitz isomorphisms. Then we prove in Theorem \ref{Gor2} a version of the Gorelik principle which is a variant of Theorem 3.8 in \cite{GodefroyLancienZizler2014} stated in terms of coarse Lipschitz equivalences instead of net equivalences of Banach spaces.

\smallskip
Section \ref{auscle} is devoted to the study of the preservation of the asymptotic uniform smoothness under Lipschitz isomorphisms and coarse Lipschitz equivalences. First we recall the definitions of the relevant moduli and their relationships. The stability of the existence of an equivalent asymptotically uniformly smooth norm was proved in \cite{GodefroyKaltonLancien2001} in the separable case. We take in Theorem \ref{Lipschitz-equivalence} the opportunity to detail its proof in the non separable case that we have not found in the literature. In Theorem \ref{CLequivalence} we detail a precise quantitative version of the stability of asymptotically uniformly smooth renormings under coarse Lipschitz equivalences, again in the general case. This result was mentioned in \cite{GodefroyLancienZizler2014} with only a very brief outline of the proof. Moreover the details of this proof will be used in our last section.

\smallskip
Finally, in section \ref{nacl} (Theorem \ref{NA}), we extend Godefroy's result to our setting of norm attaining coarse Lipschitz equivalences and we give examples of situations when it can be properly stated that the set of norm attaining coarse Lipschitz maps between two Banach spaces $X$ and $Y$ is not dense in the quotient space $CL(X,Y)$.

\section{Norm attaining coarse Lipschitz maps}\label{clm}

\begin{Def} Let $(M,d)$ and $(N,\delta)$ be two metric
spaces and a map $f:M\to N$. If $(M,d)$ is unbounded, we define
$$\forall s>0,\ \ Lip_s(f)=\sup\Big\{\frac{\delta((f(x),f(y))}{d(x,y)},\ d(x,y)\ge s\Big\}\ \
{\rm and}\ \ Lip_\infty(f)=\inf_{s>0}Lip_s(f).$$
Then $f$ is said to be {\it coarse Lipschitz} if
$Lip_\infty(f)<\infty$.\\
The set of coarse Lipschitz maps from $M$ to $N$ is denoted $\cal C\cal L(M,N)$.
\end{Def}

The following equivalent formulations are easy to verify.

\begin{Prop}\label{CL} Let $X$ and $Y$ be two Banach spaces and let $f:X\to Y$ be a mapping. Then the following assertions are equivalent.

\smallskip (i)  The map $f$ is coarse Lipschitz.

\smallskip (ii) There exist $A$ and $\theta$ in $[0,+\infty)$ such that
$$\forall x,x'\in X \ \ \|x-x'\|\ge \theta \Rightarrow \|f(x)-f(x')\|\le A\|x-x'\|.$$

\smallskip (iii) There exist $A$ and $B$ in $[0,+\infty)$ such that
$$\forall x,x'\in X \ \ \|f(x)-f(x')\|\le A\|x-x'\|+B.$$
\end{Prop}

Note that in the above statement, $Lip_\infty(f)$ coincide with the infimum of all $A\ge 0$ such that (ii) is satisfied for some $\theta\ge 0$ and also with the infimum of all $A\ge0$ such that (iii) is satisfied for some $B\ge 0$.

\medskip Suppose now that $(M,d)$ is an unbounded metric space and $(Y,\|\ \|_Y)$ is a Banach space. Then it is easy to see that $\cal C\cal L(M,Y)$ is a vector space on which $Lip_\infty$ is a semi-norm that we shall also denote $\|\ \|_{CL(M,Y)}$ or simply $\|\ \|_{CL}$ if no confusion is possible. Then we denote $\cal N(M,Y)=\{f\in \cal C\cal L(M,Y),\ Lip_\infty(f)=0\}$ and $CL(M,Y)$ the quotient space $\cal C\cal L(M,Y)/\cal N(M,Y)$. The semi-norm $Lip_\infty$ induces a norm on $CL(M,Y)$ that will also be denoted $Lip_\infty$, $\|\ \|_{CL(M,Y)}$ or $\|\ \|_{CL}$. We shall try to avoid as much as possible the confusion between elements of $\cal C\cal L(M,Y)$ and elements of $CL(M,Y)$.

\medskip We now introduce the notion of norm attaining coarse Lipschitz maps.
\begin{Def} Let $(M,d)$ be an unbounded metric space and $(Y,\|\ \|_Y)$ a Banach space. Assume that $f:M \to Y$ is coarse Lipschitz. We say that $f$ attains its norm in the direction $y\in S_Y$ if there exists a sequence of pairs of distinct points $(s_n,t_n)$ in $M$ such that
$$\lim_{n\to \infty}d(s_n,t_n)=+\infty\ \ {\rm and}\ \ \lim_{n\to \infty}\frac {f(t_n)-f(s_n)}{d(s_n,t_n)}=y\,Lip_\infty(f).$$
\end{Def}

\noindent {\bf Remark.} Note that the above definition is only interesting when $Lip_\infty(f)\neq 0$, that is when $f\neq 0$ in the quotient space $CL(M,Y)$.\\
Note also that if $f\in \cal C\cal L(M,Y)$ attains its norm in the direction $y\in S_Y$ and $g:M\to Y$ is such that $Lip_\infty(f-g)=0$, then $g$ also attains its norm in the direction $y$. Therefore, this notion is well defined for an element $f$ of the quotient space $CL(M,Y)$.

\section{Coarse Lipschitz equivalence of metric spaces}\label{cle}

\begin{Def} Let $(M,d)$ and $(N,\delta)$ be two unbounded metric spaces and $f:M\to N$ be a coarse Lipschitz map. We say that $f$ is a {\it coarse Lipschitz equivalence} from $M$ to $N$, if there exists a coarse Lipschitz map $g:N\to M$ and a constant $C\ge 0$ such that
$$\forall x\in M\ \ d\big((g\circ f)(x),x\big)\le C\ \ {\rm and} \ \ \forall y\in N\ \ \delta\big((f\circ g)(y),y\big)\le C.$$
We denote $\cal C \cal L \cal E(M,N)$ the set of coarse Lipschitz equivalences from $M$ to $N$. If $\cal C \cal L \cal E(M,N)$ is non empty, we say that $M$ and $N$ are {\it coarse Lipschitz equivalent} and denote $M \buildrel {CL}\over {\sim} N$.
\end{Def}

This notion of coarse Lipschitz equivalent metric spaces is exactly the same as the notion of quasi-isometric metric spaces introduced by Gromov in \cite{Gromov1987} (see also the book \cite{GhysDelaHarpe} by E. Ghys and P. de la Harpe).

\medskip

\noindent{\bf Remark.} It is easy to check, for instance using the characterization (iii) in Proposition \ref{CL}, that $\buildrel {CL}\over {\sim}$ is an equivalence relation between Banach spaces.

\medskip We now turn to the notion of net in a metric space.

\begin{Def} Let $0<a\le b$. An $(a,b)$-net in the metric space $(M,d)$ is a subset $\cal M$ of $M$ such that for every $z\neq z'$ in $\cal M$,  $d(z,z')\ge a$ and for
every $x$ in $M$, $d(x,\cal M)< b$.\\
Then a subset $\cal M$ of $M$ is a {\it net} in $M$ if it is an $(a,b)$-net for some $0<a\le b$.
\end{Def}

Let us now give two technical equivalent formulations of the notion of coarse equivalence between Banach spaces, that we shall use later. The main result, which is the fact that (ii) implies (iii) is essentially contained in the proof of Theorem 3.8 in \cite{GodefroyLancienZizler2014}.

\begin{Prop}\label{CLE} Let $X$ and $Y$ be two Banach spaces and let $f:X\to Y$ be a coarse Lipschitz map. The following assertions are equivalent.

\smallskip (i) The map $f$ belongs to $\cal C \cal L \cal E(X,Y)$

\smallskip (ii) There exist $A_0>0$ and $K\ge 1$ such that for all $A\ge A_0$ and all maximal $A$-separated subset $\cal M$ of $X$, $\cal N=f(\cal M)$ is a net in $Y$ and
$$\forall x,x'\in \cal M\ \ \ \frac{1}{K}\|x-x'\|\le \|f(x)-f(x')\|\le K\|x-x'\|.$$

\smallskip (iii) There exist two \underline{continuous} coarse Lipschitz maps $\varphi:X\to Y$ and $\psi:Y\to X$ and a constant $C\ge 0$ such that $\|\varphi(x)-f(x)\|\le C$ for all $x$ in $X$ and
$$\forall x\in X\ \ \|(\psi\circ \varphi)(x)-x\|\le C\ \ {\rm and} \ \ \forall y\in Y\ \ \|(\varphi\circ \psi)(y)-y\|\le C.$$

\end{Prop}

\begin{proof} $(i)\Rightarrow (ii)$. Assume that there exist $g:Y\to X$ and constants $C,D,M>0$ such that
$$\forall x\in X\ \ \|(g\circ f)(x)-x\|\le C,\ \ \ \ \forall y\in Y\ \ \|(f\circ g)(y)-y\|\le C.$$
and
$$\forall x,x'\in X\ \ \|f(x)-f(x')\|\le D+M\|x-x'\|,$$
$$\forall y,y'\in Y\ \ \|g(y)-g(y')\|\le D+M\|y-y'\|.$$
Let $A_0= (2C+D)(M+1)$, $A\ge A_0$ and $\cal M$ be a maximal $A$-separated subset of $X$. Note that $\cal M$ is a $(A,A)$-net of $X$. Let now $x\neq x' \in \cal M$, $y=f(x)$ and $y'=f(x')$. Then $$\|f(x)-f(x')\|\le D+M\|x-x'\|\le A+M\|x-x'\|\le (M+1)\|x-x'\|.$$ On the other hand $\|g(y)-x\|\le C$ and $\|g(y')-x'\|\le C$, which implies that \\
$\|g(y)-g(y')\|\ge \|x-x'\|-2C$ and therefore
$$\|x-x'\|\le 2C+D+M\|y-y'\|\le \frac{A}{M+1}+M\|y-y'\|\le \frac{\|x-x'\|}{M+1}+M\|y-y'\|.$$
It follows that $\|x-x'\|\le (M+1)\,\|y-y'\|$. So $f$ is a Lipschitz isomorphism from $\cal M$ onto $\cal N=f(\cal M)$ and $K=M+1$ satisfies the required inequalities. In particular $\cal N$ is $a$-separated, with $a=A(M+1)^{-1}$.\\ Finally let $z\in Y$. There exists $x\in \cal M$ such that $\|x-g(z)\|\le A$. Then $$\|f(x)-z\|\le \|f(x)-f(g(z))\|+C\le D+MA+C=b.$$
We have shown that $\cal N$ is an $(a,b)$-net in $Y$, which finishes the proof of this implication.

\medskip
$(ii)\Rightarrow (iii)$ For $A\ge A_0$, we pick $(x_i)_{i\in I}$ a maximal $A$-separated subset of $X$. Note that $(x_i)_{i\in I}$ is an $(A,A)$-net in $X$. For $i\in I$, let $y_i=f(x_i)$. Then, by assumption, $(y_i)_{i\in I}$ is an $(a,b)$-net in $Y$, for some $0<a\le b$, and we have
$$\forall i,j\in I\ \ \ \frac1K \|x_i-x_j\|\le \|y_i-y_j\| \le K\,\|x_i-x_j\|.$$
Then we can find  a continuous partition of unity $(f_i)_{i\in I}$ subordinated to the open cover $(B_X(x_i,A))_{i\in I}$ of $X$ and a continuous partition of unity $(g_i)_{i\in I}$ subordinated to the open cover $(B_Y(y_i,b))_{i\in I}$ of $Y$ and we set
$$\forall x\in X\ \ \varphi(x)=\sum_{i\in I}f_i(x)\,y_i\ \ {\rm and}\ \ \forall y \in Y\ \ \psi(y)=\sum_{i\in I}g_i(y)\,x_i.$$
Note first that $\varphi$ and $\psi$ are continuous.\\
Let $x\in X$ and pick $i\in I$ such that $\|x-x_i\|\le A$. Now, if $f_j(x)\neq 0$, then $\|x-x_j\|\le A$ and $\|x_i-x_j\|\le 2A$. It follows that
$$\|\varphi(x)-y_i\|=\|\sum_{j,f_j(x)\neq 0} f_j(x)\,(y_j-y_i)\|\le 2AK.$$
Let now $x'\in X$ and $j\in I$ so that $\|x'-x_j\|\le A$. Then we have
$$\|\varphi(x)-\varphi(x')\|\le 4AK+\|y_i-y_j\|\le 4AK+K\|x_i-x_j\|\le 6AK+K\|x-x'\|.$$
This shows that $\varphi$ is coarse Lipschitz and $Lip_\infty(\varphi)\le K$ and a similar proof yields that the same is true for $\psi$.\\
For $x\in X$, pick again $i\in I$ such that $\|x-x_i\|\le A$. If $g_j(\varphi(x))\neq 0$, then $\|\varphi(x)-y_j\|\le b$ and $\|y_i-y_j\|\le \|\varphi(x)-y_i\|+\|\varphi(x)-y_j\|\le 2AK+b$. Therefore
$$\|\psi(\varphi(x))-x_i\|=\big\|\sum_{j,g_j(\varphi(x))\neq 0} g_j(\varphi(x))\,(x_j-x_i)\big\|\le K(2AK+b).$$
Finally, we get that
$$\|\psi(\varphi(x))-x\|\le \|\psi(\varphi(x))-x_i\|+\|x-x_i\|\le K(2AK+b)+A=C_1.$$
Similarly, we get that there exists $C_2\ge 0$ such that for all $y\in Y$, $\|\varphi(\psi(y))-y\|\le C_2$.\\
Finally, recall that $f$ is coarse Lipschitz. So, there exist $D,E\ge 0$ such that for all $x,x'\in X$, $\|f(x)-f(x')\|\le D\|x-x'\|+E$. Since
$$\forall x\in X\ \ \ \varphi(x)-f(x)=\sum_{j,f_j(x)\neq 0} f_j(x)\,(f(x_j)-f(x)),$$
and $\|x_j-x\|\le A$, whenever $f_j(x)\neq 0$, we obtain that
$$\forall x\in X,\ \ \|\varphi(x)-f(x)\|\le DA+E=C_3.$$
We conclude the proof of this implication by taking $C=\max\{C_1,C_2,C_3\}$.

\medskip $(iii)\Rightarrow (i)$ is clear.

\end{proof}

\noindent{\bf Remark.} The main information of Proposition \ref{CLE} is that for any $f$ in $\cal C \cal L \cal E(X,Y)$, there exists $\varphi$ which is a continuous representative of the equivalence class of $f$ in $CL(X,Y)$ and also a coarse Lipschitz equivalence with a continuous ``coarse Lipschitz inverse'' $\psi$. This will be crucial when we shall apply the Gorelik principle whose proof is based on Brouwer's fixed point theorem.

\medskip

Let us notice that, using for instance the characterization (ii) of Proposition \ref{CLE}, the following is immediate.

\begin{Cor}\label{homogeneity} Let $X$, $Y$ be two Banach spaces and $f\in \cal C \cal L \cal E(X,Y)$. Then, for any $\lambda \neq 0$, $\lambda f\in \cal C \cal L \cal E(X,Y)$.
\end{Cor}

\section{On the completeness of CL(X,Y)}\label{completeness}

\begin{Def} Let $X$ and $Y$ be two Banach spaces and $\cal M$ be a net in $X$. We say that  $(\cal M,X,Y)$ has the {\it Lipschitz extension property} if any Lipschitz function from $\cal M$ to $Y$ admits a Lipschitz extension from $X$ to $Y$. We say that the pair $(X,Y)$ has the {\it net extension property} (in short NEP) if there exists a net $\cal M$ in $X$ such that $(\cal M,X,Y)$ has the Lipschitz extension property
\end{Def}

\begin{Lem}\label{uniform} Assume that $X$ and $Y$ are Banach spaces and $\cal M$ is a net in $X$ such that $(\cal M, X,Y)$ has the Lipschitz extension property, then there exists $\lambda\ge 1$ such that any Lipschitz function $f:\cal M \to Y$ admits an extension $g:X\to Y$ with $Lip\,(g)\le \lambda Lip\,(f)$.
\end{Lem}

\begin{proof} We may and do assume that $0\in \cal M$ and $f(0)=0$. Then the conclusion follows from a straightforward application of the open mapping theorem to the restriction operator to $\cal M$ defined from $Lip_0(X,Y)$ onto $Lip_0(\cal M,Y)$, where $Lip_0(X,Y)$ is the Banach space of all Lipschitz functions from $X$ to $Y$ that vanish at 0 equipped with the norm $\|f\|_L=Lip\,(f)$.
\end{proof}

\begin{Def} Let $X$ and $Y$ be two Banach spaces and let $\mu\ge 1$. We say that  $(X,Y)$ has the {\it $\mu$-Lipschitz representation property} if (in short $\mu$-LRP) if for any $f\in \cal C \cal L(X,Y)$ and any $c>Lip_\infty(f)$, there exists $g\in Lip_0(X,Y)$ so that $Lip\,(g)< \mu c$ and $f-g$ is bounded.
\end{Def}

\begin{Prop}\label{lipschitzrep} Assume that $X$ and $Y$ are Banach spaces such that $(X,Y)$ has the net extension property. Then there exists $\mu\ge 1$ such that $(X,Y)$ has the $\mu$-LRP.
\end{Prop}

\begin{proof} Let $f\in \cal C \cal L(X,Y)$ such that $Lip_\infty(f)<c$. Pick $\cal M$ be a net in $X$ and $\lambda \ge 1$ such that the conclusion of Lemma \ref{uniform} is satisfied. It follows from an easy change of variable argument that for any $A\ge 1$, the net $A\cal M$ also satisfies  the conclusion of Lemma \ref{uniform} with the same constant $\lambda$. Then for $A$ large enough, the restriction of $f$ to $A\cal M$ is $c$-Lipschitz. So it admits an extension $g:X\to Y$ such that $g$ is $\lambda c$-Lipschitz. Since $f$ and $g$ are both coarse Lipschitz and coincide on a net, it is not difficult to see that $f-g$ is bounded. By adding a constant to $g$, we may also assume that $g(0)=0$, which concludes the proof.
\end{proof}

\noindent {\bf Remark.} We do not know if the converse of this last proposition is true. However, it is not difficult to check that the existence of $\mu\ge 1$ such that $(X,Y)$ has the $\mu$-LRP is equivalent to the existence of $\lambda\ge 1$ such that $(X,Y)$ has the $\lambda$-ANEP. Here, $\lambda$-ANEP stands for $\lambda$-almost net extension property, which is formally weaker than NEP and has the following ad'hoc meaning: there exists a net $\cal M$ in $X$ such that for any Lipschitz function $f:\cal M \to Y$, there exists $g:X\to Y$ Lipschitz so that $Lip\,(g)\le \lambda Lip\,(f)$ and $f-g$ is bounded on $\cal M$.

\begin{Prop}\label{complete} Assume that $X$ and $Y$ are Banach spaces such that $(X,Y)$ has the $\mu$-LRP for some $\mu\ge 1$. Then $(CL(X,Y),\|\ \|_{CL})$ is a Banach space
\end{Prop}

\begin{proof} Let $(f_n)_{n=1}^\infty$ be a sequence in $\cal C \cal L(X,Y)$ such that $\sum_{n=1}^\infty \|f_n\|_{CL}<\infty$. Then for any $n$ in $\Ndb$, there exists $g_n\in Lip_0(X,Y)$ such that $g_n$ belongs to the equivalence class of $f_n$ and $Lip\,(g_n)\le \mu \|f_n\|_{CL}+2^{-n}$. Then using the completeness of $(Lip_0(X,Y),\|\ \|_L)$ we get that there exists $g\in Lip_0(X,Y)$ such that
$$\lim_{N\to \infty}\|g-\sum_{n=1}^N g_n\|_L=0.$$ It follows that $\lim_{N\to \infty}\|g-\sum_{n=1}^N f_n\|_{CL}=0$, which concludes our proof.
\end{proof}

\noindent {\bf Remark.} We do not know if $(CL(X,Y),\|\ \|_{CL})$ is a Banach space without any assumption on the Banach spaces $X$ and $Y$. We conjecture that it is not the case, but a counterexample still has to be constructed.

\begin{Prop}\label{open} Assume that $X$ and $Y$ are Banach spaces such that $(X,Y)$ and $(Y,X)$ have the $\mu$-LRP for some $\mu\ge 1$. Then for any $f\in \cal C \cal L \cal E(X,Y)$, there exists $\eps>0$ such that $f-u\in \cal C \cal L \cal E(X,Y)$, whenever $u:X\to Y$ is such that $Lip_\infty(u)<\eps$.\\
\end{Prop}

\begin{proof} Since $f\in \cal C \cal L \cal E(X,Y)$, there exists $C\ge 1$ and $g\in \cal C \cal L \cal E(Y,X)$ so that $Lip_\infty(f)< C$, $Lip_\infty(g)< C$ and
$$\forall x \in X\ \ \|(g\circ f)(x)-x\|\le C\ \ {\rm and}\ \ \forall y\in Y\ \ \|(f\circ g)(y)-y\|\le C.$$
Let us now fix $u\in \cal C \cal L(X,Y)$ such that $Lip_\infty(u)<(\mu^2 C)^{-1}$.\\
It follows from our assumptions that there exist $K>0$, $\varphi \in Lip_0(X,Y)$, $\psi \in Lip_0(Y,X)$ and $v\in Lip_0(X,Y)$ so that $Lip\,(\varphi)<\mu C$, $Lip\,(\psi)<\mu C$, $Lip\,(v)<(\mu C)^{-1}$ and such that $f-\varphi$, $g-\psi$ and $u-v$ are bounded. Note first that it is not difficult to deduce that $\psi\circ \varphi-I_X $ is bounded on $X$ and $\varphi\circ \psi-Id_Y$ is bounded on $Y$. So let $K>0$ be such that $\|f-\varphi\|$, $\|u-v\|$, $\|g-\psi\|$, $\|\psi\circ \varphi-I_X\|$ and $\|\varphi\circ \psi-Id_Y\|$ are bounded by $K$ on their respective domains.

We now exhibit a coarse Lipschitz inverse $G$ of $f-u$ as follows. For $y\in Y$ and $x\in X$, we define $L_y(x)=\psi(y+v(x))$. Since $Lip\,(L_y)<1$, the map $L_y$ admits a unique fixed point in $X$ that we denote $G(y)$, which is thus defined by the equation
\begin{equation}\label{fixed point}
G(y)=\psi\big(y+v\circ G(y))\big).
\end{equation}
Classical elementary manipulations of the above equation yield that $G$ is Lipschitz and more precisely that $Lip\,(G)\le Lip\,(\psi)\big(1-Lip\,(v)Lip\,(\psi)\big)^{-1}$. It remains to show that $(f-u)\circ G -I_Y$ and $G\circ (f-u)-I_X$ are bounded. Since $G$ is Lipschitz, it is enough to show that $(\varphi - v)\circ G -I_Y$ and $G\circ (\varphi-v)-I_X$ are bounded.\\
Let us first fix $y\in Y$. Then using (\ref{fixed point}) we get
\[
 \begin{aligned}
 \|(\varphi-v)\circ G(y)-y\|&=\|\varphi \circ \psi \big(y+v\circ G(y)\big) - v \circ \psi\big(y+v\circ G(y)\big)-y\|\\
 &\le K+\|y+v\circ G(y) -v\circ G(y)-y\|=K.
  \end{aligned}
 \]
Consider $x\in X$. Then
\[
 \begin{aligned}
 \|G \circ(\varphi-v)(x)-x\|&\le\|\psi\big((\varphi-v)(x)+v\circ G\big((\varphi-v)(x)\big)-\psi \circ \varphi(x)\|+K\\
 &\le Lip\,(\psi)\|v\circ G\big((\varphi-v)(x)\big)-v(x)\|+K\\
 &\le Lip\,(\psi)Lip\,(v)\|G \circ(\varphi-v)(x)-x\|+K.
  \end{aligned}
 \]
It follows that
$$\|G \circ(\varphi-v)(x)-x\|\le K\big((1-Lip\,(v)Lip\,(\psi)\big)^{-1}.$$
We have proved that $f-u \in \cal C \cal L \cal E(X,Y)$.
\end{proof}

\noindent{\bf Remarks.}

\smallskip
Note that for $X$ and $Y$ Banach spaces and $f:X\to Y$ coarse Lipschitz, $f\in \cal C \cal L \cal E(X,Y)$ if and only if all the elements of its equivalence class in $CL(X,Y)$ belong to $\cal C \cal L \cal E(X,Y)$ (this is a consequence of Proposition \ref{CLE}). So, in the particular situation described in Proposition \ref{open}, we can denote $CLE(X,Y)$ the set of equivalent classes of elements of $\cal C \cal L \cal E(X,Y)$ and state that it is open in the quotient space $CL(X,Y)$.

\smallskip
In this work we have chosen to follow Gromov's definition for the invertible elements of $\cal C \cal L (X,Y)$. One of the advantages of this definition is to coincide with the notion of net equivalence for Banach spaces. However, in pursuing the study of our normed quotient space, it could be more natural to say that $f\in  \cal C \cal L (X,Y)$ is ``invertible'' if there exists $g\in \cal C \cal L (Y,X)$ such that $Lip_\infty\big((f\circ g)-Id_Y\big)= Lip_\infty\big((g\circ f)-Id_X\big)=0$.

\section{Background on the Gorelik principle.}\label{Gorelik}

The tool that we shall now
describe is the Gorelik principle. It was initially devised by
Gorelik in \cite{Gorelik1994} to prove that $\ell_p$ is not uniformly
homeomorphic to $L_p$, for $1<p<\infty$. Then it was developed by
Johnson, Lindenstrauss and Schechtman \cite{JohnsonLindenstraussSchechtman1996} to prove that for
$1<p<\infty$, $\ell_p$ has a unique uniform structure. We now recall the crucial ingredient in the proof of the Gorelik Principle (see step (i) in the proof of Theorem 10.12 in \cite{BenyaminiLindenstrauss2000}). This statement relies on Brouwer's fixed point
theorem and on the existence of Bartle-Graves continuous selectors. We refer the reader to \cite{AlbiacKalton2016} or \cite{BenyaminiLindenstrauss2000} for its proof.

\begin{Prop}\label{G} Let $X_0$ be a finite-codimensional subspace
of  a Banach space $X$ and let $0<c<d$. Then,  there exists a compact subset $A$ of
$dB_X$ such that for every continuous map $\phi:A\to X$ satisfying $\|\phi(a)-a\|\le c$ for
all $a\in A$, we have that $\phi(A)\cap X_0\neq \emptyset$.
\end{Prop}

Let us now state the Gorelik principle as it can be found in \cite{AlbiacKalton2016}, \cite{BenyaminiLindenstrauss2000} or \cite{GodefroyKaltonLancien2001}.

\begin{Thm}\label{Gor1} Let $X$ and $Y$ be
two Banach spaces and let $f$ be a homeomorphism from $X$ onto $Y$ whose
inverse is uniformly continuous.  Let $b,d>0$ so that $\omega(f^{-1},b)<d$, where $\omega(f^{-1},.)$ is the modulus of uniform continuity of $f^{-1}$.
Assume that $X_0$ is a closed finite codimensional subspace of $X$.  Then
there exists a compact subset $K$ of $Y$ so that
$$ bB_Y \subset K+ f(2d B_{X_0}).$$
In particular, if $f$ is a Lipschitz isomorphism such that $Lip(f)\le 1$ and $Lip(f^{-1})\le M$, the condition $\omega(f^{-1},b)<d$ is satisfied when $Mb<d$.

\end{Thm}

We will now state a version of the Gorelik principle that will be used to study coarse equivalent Banach spaces. For the sake of completeness we shall reproduce the proof that can be found in \cite{GodefroyLancienZizler2014} (Theorem 3.8) with more attention given on keeping optimal estimates and with slightly weaker assumptions.

\begin{Thm}\label{Gor2} Let $X$ and $Y$ be two Banach spaces. Assume that $f:X\to Y$ and $g:Y\to X$ are continuous, and  that there exist constants $C,D,M > 0$ such that
$$\forall y,y'\in Y \ \ \|g(y)-g(y')\|\le D+ M\|y-y'\|$$
and
$$\forall x\in X\ \ \|(g\circ f)(x)-x\|\le C\ \ {\rm and} \ \ \forall y\in Y\ \ \|(f\circ g)(y)-y\|\le C.$$
Let $\lambda<1$. Then for any $\alpha\ge \alpha_0=2(C+D)(1-\lambda)^{-1}$ and any finite codimensional subspace $X_0$ of $X$, there is a compact subset $K$ of $Y$ so that
$$\frac{\lambda\alpha}{M}B_Y\subset K+CB_Y+f(2\alpha B_{X_0}).$$
\end{Thm}

\begin{proof} Let $\mu=\frac{1+\lambda}{2}$, $\alpha_0=\frac{C+D}{\mu-\lambda}=\frac{2(C+D)}{1-\lambda}$ and $\alpha\ge \alpha_0$. Let also $X_0$ be a finite codimensional subspace of $X$.\\
It follows from Proposition \ref{G} that there exists a compact subset $A$ of
$\alpha B_X$ such that for every continuous map $\phi:A\to X$ satisfying $\|\phi(a)-a\|\le \mu\alpha$ for all $a\in A$, we have that $\phi(A)\cap X_0\neq \emptyset$.\\
Consider now $y\in \frac{\lambda\alpha}{M}B_Y$ and define $\phi:A\to X$ by
$\phi(a)=g(y+f(a))$. Then $\phi$ is clearly continuous and
$$\forall a\in A, \ \ \|\phi(a)-a\|\le C+\|g(y+f(a))-g(f(a))\|\le C+D+M\|y\|\le C+D+\lambda \alpha\le \mu\alpha.$$
Hence, there exists $a\in A$ so that $\phi(a)\in X_0$. Since $\|a\|\le \alpha$ and $\|\phi(a)-a\|\le \mu\alpha$, we have that $\phi(a)\in 2\alpha B_{X_0}$.\\
Finally, we use the fact that $\|(f\circ g)(y+f(a))-(y+f(a))\|\le C$ to conclude that $y\in K+CB_Y+f(2\alpha B_{X_0})$, where $K=-f(A)$ is a compact subset of $Y$.
\end{proof}

\noindent {\bf Remark.} Note that in the above result we have not assumed that $f$ is coarse Lipschitz.

\section{Asymptotic uniform smoothness and coarse Lipschitz equivalence}\label{auscle}

We now recall the definitions of the modulus of asymptotic uniform smoothness of a norm and the modulus of weak$^*$ asymptotic uniform convexity of a dual norm. They are due to V. Milman \cite{Milman1971} and we follow the notation from \cite{JohnsonLindenstraussPreissSchechtman2002}. So let $(X,\|\ \|)$ be a Banach space. For $t>0$, and $x\in S_X$  we define
$$\overline{\rho}_X(x,t)=\inf_{Y}\sup_{y\in S_Y}(\|x+t y\|-1),$$
where $Y$ runs through all closed subspaces of $X$ of finite codimension. Then
$$\overline{\rho}_X(t)=\sup_{x\in S_X}\overline{\rho}_X(x,t).$$
The norm $\|\ \|$ is said to be
{\it asymptotically uniformly smooth} (in short AUS) if
$$\lim_{t \to 0}\frac{\overline{\rho}_X(t)}{t}=0.$$
We say that the norm $\|\ \|$ is {\it asymptotically uniformly flat} if
$$\exists t_0\in (0,+\infty)\ \ \ \forall t\in [0,t_0]\ \ \overline{\rho}_X(t)=0.$$
Now, for $t>0$, and $x^*\in S_{X^*}$  we define
$$ \overline{\theta}_X(x^*,t)=\sup_{E}\inf_{y^*\in S_{E^\perp}}(\|x^*+ty^*\|-1),$$
where $E$ runs through all finite dimensional subspaces of $X$. Then
$$\overline{\theta}_X(t)=\inf_{x^*\in S_{X^*}}\overline{\theta}_X(x^*,t).$$
The norm of $X^*$ is said to be {\it weak$^*$ asymptotically uniformly convex} (in short w$^*$-AUC) if
$$\forall t>0\ \ \ \ \overline{\theta}_X(t)>0.$$

The duality between these two moduli is now well understood. The following complete and precise statement is taken from \cite{DilworthKutzarovaLancienRandri2017} Proposition 2.1.

\begin{Prop}\label{Young} Let $X$ be a Banach space and $0<\sigma,\tau<1$.\\
\smallskip
(a) If $\overline{\rho}_X(\sigma)<\frac{\sigma\tau}{6}$, then  $\overline{\theta}_X(\tau)>  \frac{\sigma\tau}{6}$.\\
\smallskip
(b) If $\overline{\theta}_X(\tau)> \sigma\tau$, then $\overline{\rho}_X(\sigma)<\sigma\tau$
\end{Prop}

As an immediate consequence we have that $\|\ \|_X$ is AUS if and only if $\|\ \|_{X^*}$ is w$^*$-AUC.

\medskip Let us also detail a few other classical consequences. First we recall that for a function $f$ which is continuous monotone non decreasing on $[0,1]$ and such that $f(0)=0$, its dual Young function is denoted $f^*$ and  defined by
$$\forall s\in [0,1]\ \ \ f^*(s)=\sup\{st-f(t),\ t\in [0,1]\}.$$
As a corollary of the previous proposition we obtain.

\begin{Cor}\label{Young2} Let $X$ be a Banach space. Then
$$\forall s\in [0,1]\ \ \ (\overline{\theta}_X)^*(s)\ge \overline{\rho}_X\big(\frac{s}{2}\big)\ \ \ {\text and}\ \ \ (\overline{\theta}_X)^*\big(\frac{s}{6}\big)\le \overline{\rho}_X(s).$$
\end{Cor}
\begin{proof} Consider first $t=\frac2s\overline{\rho}_X(\frac{s}{2})\in [0,1]$. Then $\overline{\rho}_X(\frac{s}{2})=\frac{s}{2}t$. So it follows from Proposition \ref{Young} (b)  that $\overline{\theta}_X(t)\le \frac{s}{2}t$. Therefore $(\overline{\theta}_X)^*(s)\ge st-\overline{\theta}_X(t)\ge \frac{s}{2}t=\overline{\rho}_X(\frac{s}{2}).$

Assume now that $(\overline{\theta}_X)^*(\frac{s}{6}) > \overline{\rho}_X(s).$ Then there exists $t\in [0,1]$ such that\\
$\frac{s}{6}t-\overline{\theta}_X(t)>\overline{\rho}_X(s)$. Thus $\overline{\theta}_X(t)<\frac{s}{6}t-\overline{\rho}_X(s)\le\frac{s}{6}t$. It now follows from Proposition \ref{Young} (a) that $\overline{\rho}_X(s)\ge \frac{s}{6}t$. But this implies that $\overline{\theta}_X(t)<0$, which is impossible.
\end{proof}

The following theorem states that the existence of an asymptotically  uniformly smooth norm is stable under Lipschitz isomorphisms and appeared first in \cite{GodefroyKaltonLancien2001}, in a separable setting. Its proof can also be found in the recent textbook \cite{AlbiacKalton2016} (see paragraph 14.6). The general case can be deduced by routine arguments of separable saturation and separable determination of the moduli. However, we shall detail here the direct proof in the general case. The only modification is that we deal with the definition of the asymptotic moduli instead of using weak$^*$-null or weakly null sequences.

\begin{Thm}\label{Lipschitz-equivalence} Let $X$ and $Y$ be two Banach spaces and assume that $f:X\to Y$ is a bijection such that $Lip(f)\le 1$ and $Lip(f^{-1})\le M$. Then there exists an equivalent norm $|\ |$ on $Y$ such that $\|\ \|_Y\le |\ |\le M\|\ \|_Y$ and
$$\forall t\in [0,1],\ \ \ \overline{\theta}_{|\ |}(t)\ge \overline{\theta}_X\big(\frac{t}{4M}\big).$$
\end{Thm}

\begin{proof} Let
$$C=\overline{\rm conv}\,\Big\{\frac{f(x)-f(x')}{\|x-x'\|},\ x\neq x' \in X\Big\}.$$
Clearly, $C$ is closed convex symmetric and $C\subset B_Y$. Let now $y\in Y$ such that $\|y\|=\frac{1}{M}$. For $t\in [0,+\infty)$, denote $x_t=f^{-1}(ty)$. We have that $\|x_1-x_0\|\le 1$ and $\|x_M-x_0\|\ge 1$. So, there exists $t\in [1,M]$ such that $\|x_t-x_0\|=1$. It follows that $ty\in C$. Since $C$ is convex and symmetric, we deduce that $\frac1M B_Y\subset C$. So, if we denote $|\ |$ the Minkowski functional of $C$, we have that $|\ |$ is an equivalent norm on $Y$ such that $\|\ \|_Y\le |\ |\le M\|\ \|_Y$. Its dual norm is given by
$$\forall y^*\in Y^*\ \ \ |y^*|=\sup\left\{\frac{\langle y^*,f(x)-f(x')\rangle}{\|x-x'\|}\ \ \ x\neq x'\right\}.$$

Let $t\in (0,1]$ and assume as we may that $\overline{\theta}_X\big(\frac{t}{4M}\big)>0$. So let $y^*\in Y^*$ such that $|y^*|=1$ and $\eta >0$. We can pick $x\neq x'\in X$ such that $$\langle y^*,f(x)-f(x')\rangle \ge (1-\eta)\|x-x'\|.$$
We may assume that $x'=-x$ and $f(x')=-f(x)$, so that we have
\begin{equation}\label{eq1}
\langle y^*,f(x)\rangle \ge (1-\eta)\|x\|.
\end{equation}
Pick $0<\delta <\overline{\theta}_X(\frac{t}{4M})$. It follows from statement (b) in Proposition \ref{Young} that $\overline{\rho}_X(\frac{4M\delta}{t})<\delta$. So, there exists a finite codimensional subspace $X_0$ of $X$ such that
\begin{equation}\label{eq2}
\forall z \in \frac{4M\delta \|x\|}{t}\,B_{X_0},\ \ \ \|x+z\|\le (1+\delta)\|x\|.
\end{equation}
Pick $b<\frac{4\delta \|x\|}{t}$. It now follows from the Gorelik principle (Theorem \ref{Gor1}) that there exists a compact subset $K$ of $Y$ such that
\begin{equation}\label{eq3}
bB_Y\subset K+f\Big(\frac{8M\delta \|x\|}{t}B_{X_0}\Big).
\end{equation}
Fix now $\eps>0$, consider a finite $\eps$-net $F$ of $K$ and denote $E$ the finite dimensional subspace of $Y$ spanned by $F \cup \{f(x)\}$. For any $z^*\in E^\perp$ such that $|z^*|=t$, we have $\|z^*\|\ge t$ and, if $\eps>0$ was initially chosen small enough, by (\ref{eq3}) we deduce that
\begin{equation}\label{eq4}
\exists z \in \frac{8M\delta \|x\|}{t}B_{X_0} \ \ \ \langle z^*, -f(z)\rangle \ge (b-\eta)t.
\end{equation}
It now follows from the fact that $|y^*|=1$ and (\ref{eq2}) that
$$\langle y^*,f(x)+f(z)\rangle=\langle y^*,f(z)-f(x')\rangle \le \|x'-z\|\le (1+\delta)\|x\|.$$
Then (\ref{eq1}) implies that $\langle y^*,f(z)\rangle \le (\delta+\eta)\|x\|$. Combining this last inequality with the fact that $\langle z^*,f(x)\rangle =0$ and (\ref{eq1}), (\ref{eq2}) and (\ref{eq3}), we obtain that
$$\langle y^*+z^*,f(x)-f(z)\rangle \ge (1-\eta)\|x\|-(\delta+\eta)\|x\|+(b-\eta)t.$$
Using again the definition of $|\ |$ and (\ref{eq2}) we get
$$|y^*+z^*| \ge \Big((1-\eta)\|x\|-(\delta+\eta)\|x\|+(b-\eta)t\Big)\Big((1+\delta)\|x\|\Big)^{-1}.$$
Letting $b$ tend to $\frac{4\delta \|x\|}{t}$ and $\eta$ tend to $0$, we deduce that
$$\overline{\theta}_{|\ |}(y^*,t)\ge \frac{1+3\delta}{1+\delta}-1\ge \delta.$$
In the above estimate, which does not depend on $y^*$ in the unit sphere of $|\ |$, we let $\delta$ tend to $\overline{\theta}_X(\frac{t}{4M})$ to conclude our proof.

\end{proof}

\begin{Cor} Let $X$ and $Y$ be two Banach spaces and assume that $f:X\to Y$ is a bijection such that $Lip\,(f)\le 1$ and $Lip\,(f^{-1})\le M$. Then there exists an equivalent norm $|\ |$ on $Y$ such that $\|\ \|_Y\le |\ |\le M\|\ \|_Y$ and
$$\forall t\in [0,1]\ \ \ \overline{\rho}_{|\ |}\big(\frac{t}{48M}\big)\le \overline{\rho}_X(t).$$
\end{Cor}

\begin{proof} Let $f,g$ be continuous monotone non decreasing on $[0,1]$ with $f(0)=g(0)=0$. If there exists a constant $C\ge 1$ such that for all $t\in [0,1]$, $f(t)\ge g(t/C)$, then it is clear that for all $t\in [0,1]$, $f^*(t/C)\le g^*(t)$. Since the norm $|\ |$ given by Theorem \ref{Lipschitz-equivalence} satisfies
$$\forall t\in [0,1]\ \ \ \overline{\theta}_{|\ |}(t)\ge \overline{\theta}_X\big(\frac{t}{4M}),$$
the conclusion of the proof follows now directly from Corollary \ref{Young2}.
\end{proof}

We now turn to the study of the preservation of the modulus of weak$^*$ asymptotic uniform convexity, up to renorming, under coarse Lipschitz equivalence. The following precise quantitative statement is a slight modification of Theorem 3.12 in \cite{GodefroyLancienZizler2014}, in which the proof is only very briefly outlined. It will also be crucial for us to use the details of the construction of this equivalent norm in our last section.


\begin{Thm}\label{CLequivalence} Let $X$ and $Y$ be two Banach spaces and $M>1$. Assume that $f:X\to Y$ and $g:Y\to X$ are continuous with $Lip_\infty(f)\le 1$, $Lip_\infty(g)< M$ and that there exists a constant $C\ge 0$ such that
$$\forall x\in X\ \ \|(g\circ f)(x)-x\|\le C\ \ \text{and} \ \ \forall y\in Y\ \ \|(f\circ g)(y)-y\|\le C.$$
Then for any $\eps$ in $(0,1)$, there exists an equivalent norm $|\ |$ on $Y$ such that
$$\frac{1}{1+\eps}\|\ \|_Y\le |\ |\le M\|\ \|_Y\ \ \text{and}\ \ \forall t\in [0,1]\ \ \ \overline{\theta}_{|\ |}(t)\ge \overline{\theta}_X\big(\frac{t}{48\,M^2}\big)-\eps.$$
\end{Thm}

\begin{proof} We will adapt the proof of Theorem 5.3 in \cite{GodefroyKaltonLancien2001}.\\
For $k\in \Ndb$, we define
$$C_k=\overline{\rm conv}\,\Big\{\frac{f(x)-f(x')}{\|x-x'\|},\ \|x-x'\|\ge 2^k\Big\}.$$
Then $(C_k)_{k=1}^\infty$ is a decreasing sequence of closed convex and symmetric subsets of $Y$. Since $Lip_\infty(f)\le 1$, we have that $C_k\subset (1+\eps_k)B_Y$, where $(\eps_k)_{k=1}^\infty$ is a sequence of positive numbers tending to 0. In particular there exists $k_0\in \Ndb$ such that
$$\forall k\ge k_0\ \ \ C_k\subset (1+\frac{\eps}{16M})B_Y\subset (1+\eps)B_Y\subset 2B_Y.$$
Fix now $k\in \Ndb$, $y\in S_Y$ and denote $y_0=f(0)$. It follows easily from our assumptions that $\lim_{t\to \infty} \|g(ty)\|=\infty$. Recall also that $f(g(ty))=ty +u_t$, with $\|u_t\|\le C$. So, for $t$ large enough
$$\frac{f(g(ty))-y_0}{\|g(ty)\|}=\frac{ty}{\|g(ty)\|}+\frac{u_t-y_0}{\|g(ty)\|}\in C_k.$$
It follows from the assumption that $Lip_\infty(g)<M$ that there exist $\alpha\ge \frac1M$ and a sequence $(t_n)_n$ tending to $+\infty$ such that $\frac{t_n}{\|g(t_ny)\|}$ tends to $\alpha$. Since $C_k$ is closed, we obtain that $\alpha y\in C_k$. Finally, we use the fact that $C_k$ is convex and symmetric to deduce that $\frac1M y\in C_k$ and thus that $\frac1M B_Y \subset C_k$.\\
So, if we denote $|\ |_k$ the Minkowski functional of $C_k$, we have that for all $k\ge k_0$, $|\ |_k$ is an equivalent norm on $Y$ such that $(1+\frac{\eps}{16M})^{-1}\|\ \|_Y\le |\ |_k\le M\|\ \|_Y$. It will be useful to describe the dual norm of $|\ |_k$, also denoted $|\ |_k$, as follows
$$\forall y^*\in Y^*\ \  |y^*|_k =\sup \left\{ \frac{\langle y^*,f(x)-f(x')\rangle}{\|x-x'\|},\quad
x,x'\in X,\ \|x-x'\|\ge 2^k\right\}.$$
Note that our assumptions also imply the existence of $D\ge 0$ such that
\begin{equation}\label{eq6}
\forall y,y'\in Y \ \ \|g(y)-g(y')\|\le D+ M\|y-y'\|.
\end{equation}
This will enable us to apply the Gorelik principle as it is stated in Theorem \ref{Gor2}.

\medskip
The key lemma is the following.

\begin{Lem}\label{key} Let $t\in (0,1]$ and assume that $\overline{\theta}_X\big(\frac{t}{48M^2}\big)>0$. Let $y^*\in Y^*$ such that $ \|y^*\|\le M$,  $\eps>0$ and $k_1\in \Ndb$ such that
$$k_1\ge k_0,\ \ \ 24M^2\overline{\theta}_X\big(\frac{t}{48M^2}\big)2^{k_1}> 4(C+D)t\ \ \ {\text and}\ \ \ 2^{-k_1}(CM+1)\le \frac{\eps}{8}.$$
Then there exists a finite dimensional subspace $E$ of $Y$ so that for all $k\ge k_1$ and all $z^*\in E^\perp$ such that $\frac{t}{2}\le \|z^*\|\le tM$, we have
\begin{equation}\label{eq7}
|y^*+z^*|_k\ge 2|y^*|_{k+1}-|y^*|_k+\overline{\theta}_X\big(\frac{t}{48M^2}\big)-\frac{\eps}{2}.
\end{equation}
\end{Lem}

\begin{proof} Let $\eta =\frac{\eps}{16M}$ and
pick $0<\delta <\overline{\theta}_X(\frac{t}{48M^2})$ such that
\begin{equation}\label{eq8.0}
24M^2\delta2^{k_1}>4(C+D)t.
\end{equation}
Let $k\ge k_1$ and choose $x\neq x'\in X$ such that $\|x-x'\|\ge 2^{k+1}$ and
$$\langle y^*,f(x)-f(x')\rangle \ge (1-\eta)|y^*|_{k+1}\|x-x'\|.$$
We may assume that $x'=-x$ and $f(x')=-f(x)$, so that we have
\begin{equation}\label{eq8}
\langle y^*,f(x)\rangle \ge (1-\eta)|y^*|_{k+1}\|x\|.
\end{equation}
Since $0<\delta <\overline{\theta}_X(\frac{t}{48M^2})$. It follows from statement (b) in Proposition \ref{Young} that $\overline{\rho}_X(\frac{48M^2\delta}{t})<\delta$. So, there exists a finite codimensional subspace $X_0$ of $X$ such that
\begin{equation}\label{eq9}
\forall z \in \frac{48M^2\delta \|x\|}{t}\,B_{X_0}\ \ \ \|x+z\|\le (1+\delta)\|x\|\ \ {\rm and}\ \ \|x+z\|\ge \|x\|\ge 2^k.
\end{equation}
From (\ref{eq6}), (\ref{eq8.0}) and Theorem \ref{Gor2}, applied with $\lambda=\frac12$ and $\alpha=t^{-1}24M^2\delta\|x\|$, we infer the existence of a compact subset $K$ of $Y$ such that
$$\frac{12M\delta\|x\|}{t}B_Y\subset K+CB_Y+f\Big(\frac{48M^2\delta\|x\|}{t}B_{X_0}\Big).$$
As in the previous proof, fix $\eta'>0$, pick a finite $\eta'$-net $F$ of $K$ and let $E$ be the linear span of $F\cup\{f(x)\}$. Let now $z^*\in E^\perp$ such that $\frac{t}{2} \le \|z^*\|\le tM$. Then,
\begin{equation}\label{eq10}
\exists z \in \frac{48M^2\delta\|x\|}{t}B_{X_0}\ \ \ \langle z^*,-f(z)\rangle \ge 6M\delta \|x\| -(CM+1),
\end{equation}
if $\eta'$ was initially chosen small enough.\\
We then deduce from (\ref{eq9}) that
\begin{equation}
\langle y^*,f(x)+f(z)\rangle=\langle y^*,f(z)-f(x')\rangle \le |y^*|_k\|x'-z\|\le (1+\delta)|y^*|_k\|x\|.
\end{equation}
Thus, combining the above informations we get
$$\langle y^*+z^*,f(x)-f(z)\rangle \ge \Big(2(1-\eta)|y^*|_{k+1}-(1+\delta)|y^*|_k+6M\delta\Big)\|x\|-(CM+1).$$
Using again (\ref{eq9}) we then have
$$|y^*+z^*|_k\ge \frac{1}{1+\delta}\Big(2(1-\eta)|y^*|_{k+1}-(1+\delta)|y^*|_k+6M\delta -
2^{-k_1}(CM+1)\Big).$$
So, it follows from our initial choice of $k_1$ and the fact that $\delta\le 1$ and $M\ge 1$ that
$$|y^*+z^*|_k\ge2(1-\eta)(1-\delta)|y^*|_{k+1}-|y^*|_k+(2M+1)\delta-\frac{\eps}{8}.$$
Note that $\|y^*\|\le M$ implies that $|y^*|_{k+1}\le (1+\frac{\eps}{16M})M$ and recall that $\eta=\frac{\eps}{16M}$. So we obtain that
$$|y^*+z^*|_k\ge2(1-\eta)|y^*|_{k+1}-|y^*|_k-2M\delta+(2M+1)\delta-\frac{\eps}{8}.$$
Then with our choice of $\eta$ implies that
$$|y^*+z^*|_k\ge 2|y^*|_{k+1}-|y^*|_k+\delta-\frac{3\eps}{8}.$$
So, if $\delta$ was initially chosen close enough to $\overline{\theta}_X(\frac{t}{48M^2})$, we obtain
$$|y^*+z^*|_k\ge 2|y^*|_{k+1}-|y^*|_k+\overline{\theta}_X\big(\frac{t}{48M^2}\big)-\frac{\eps}{2}.$$

\end{proof}

\noindent {\it End of proof of Theorem \ref{CLequivalence}.} Note that a simple convexity argument shows that for any space $Z$, the function $t\mapsto t^{-1}\overline{\theta}_Z(t)$ is increasing on $(0,1]$.

Assume first that $\overline{\theta}_X\big(\frac{1}{48M^2}\big)\le\frac{\eps}{2}.$
Then for any $t \in (0,1]$ we have that
$$\overline{\theta}_Y(t)\ge 0> \overline{\theta}_X\big(\frac{1}{48M^2}\big)-\eps\ge \overline{\theta}_X\big(\frac{t}{48M^2}\big)-\eps,$$
and the original norm on $Y$ works.

Assume now that $\overline{\theta}_X\big(\frac{1}{48M^2}\big)>\frac{\eps}{2}.$ Since $\overline{\theta}_X$ is continuous, there exists $t_0\in (0,1)$ so that $\overline{\theta}_X\big(\frac{t_0}{48M^2}\big)=\frac{\eps}{2}.$ As above, we easily have that for any equivalent norm $N$ on $Y$ and any $t\in (0,t_0]$, $\overline{\theta}_N(t)\ge 0\ge \overline{\theta}_X\big(\frac{t}{48M^2}\big)-\eps$. So we only have to treat the problem for $t\in [t_0,1]$. Let us pick $k_1\in \Ndb$ satisfying the assumptions of Lemma \ref{key} for $t_0$. It then follows from the monotonicity of $t\mapsto t^{-1}\overline{\theta}_X(t)$ that the conclusion of Lemma \ref{key} applies for any $t\in [t_0,1]$ and any $k\ge k_1$.\\
Pick now $N\in \Ndb$ such that $\frac{4M}{N}<\frac{\eps}{2}$ and define
$$|y^*|=\frac1N \sum_{k=k_1+1}^{k_1+N}|y^*|_k$$ which is a dual norm on $Y^*$ with $$M^{-1}\|y^*\|\le |y^*|\le (1+\frac{\eps}{16M})\|y^*\| \le (1+\eps)\|y^*\|\le 2 \|y^*\|.$$
Let $y^*\in Y^*$, with $|y^*|=1$. It follows from Lemma \ref{key} that for any $t\in [t_0,1]$, there exists a finite dimensional subspace $E$ of $Y$ so that for all $k\in [k_1,k_1+N]$ and all $z^*\in E^\perp$ such that $|z^*|=t$, we have
$$|y^*+z^*|_k\ge 2|y^*|_{k+1}-|y^*|_k+\overline{\theta}_X\big(\frac{t}{48M^2}\big)-\frac{\eps}{2},$$
which implies, summing over $k$, that
$$|y^*+z^*|\ge |y^*|+\frac{2}{N}\big(|y^*|_{k_1+N+1}-|y^*|_{k_1+1}\big)+
\overline{\theta}_X\big(\frac{t}{48M^2}\big)-\frac{\eps}{2}.$$
Since $|y^*|=1$, we have that $\|y^*\|\le M$ and $|y^*|_{k+1}\le 2M$. So
$$|y^*+z^*|\ge |y^*|+\overline{\theta}_X\big(\frac{t}{48M^2}\big)-\frac{\eps}{2}-\frac{4M}{N}\ge |y^*|+\overline{\theta}_X\big(\frac{t}{48M^2}\big)-\eps.$$
This shows that for all $t\in [t_0,1]$, $\overline{\theta}_{|\ |}(t)\ge \overline{\theta}_X\big(\frac{t}{48M^2}\big)-\eps$ and concludes our proof.

\end{proof}

\begin{Cor}\label{aus} Let $X$ and $Y$ be two Banach spaces and $M>1$. Assume that $f:X\to Y$ and $g:Y\to X$ are continuous with $Lip_\infty(f)\le 1$, $Lip_\infty(g)< M$ and that there exists a constant $C\ge 0$ such that
$$\forall x\in X\ \ \|(g\circ f)(x)-x\|\le C\ \ {\rm and} \ \ \forall y\in Y\ \ \|(f\circ g)(y)-y\|\le C.$$
Then for any $\eps$ in $(0,1)$, there exists an equivalent norm $|\ |$ on $Y$ such that
$$\frac{1}{1+\eps}\|\ \|_Y\le |\ |\le M\|\ \|_Y\ \ {\rm and}\ \ \forall t\in [0,1]\ \ \ \overline{\rho}_{|\ |}\big(\frac{t}{576\,M^2}\big)\le \overline{\rho}_X(t)+\eps.$$
\end{Cor}

\begin{proof} Let $\varphi,\psi$ be continuous monotone non decreasing on $[0,1]$ with $\varphi(0)=\psi(0)=0$. If there exists $D\ge 1$ and $\eps >0$ such that for all $t\in [0,1]$, $\varphi(t)\ge \psi(t/D)-\eps$, then it is clear that for all $t\in [0,1]$, $\varphi^*(t/D)\le \psi^*(t)+\eps$. Then we can apply Corollary \ref{Young2} to get that if $|\ |$ is the norm given by Theorem \ref{CLequivalence}, then for all $t\in [0,1]$:
$$\overline{\rho}_{|\ |}\big(\frac{t}{576\,M^2}\big)\le
(\overline{\theta}_{|\ |})^*\big(\frac{t}{288\,M^2}\big)\le
(\overline{\theta}_{X})^*\big(\frac{t}{6}\big)+\eps \le
\overline{\rho}_X(t)+\eps.$$
\end{proof}


\section{Application to norm attaining coarse Lipschitz maps}\label{nacl}

In this section, we will extend to the setting of coarse Lipschitz maps and equivalences, the results obtained by G. Godefroy in \cite{Godefroy2016} on norm attaining Lipschitz maps. Our first result is the analogue of Theorem 3.2 of \cite{Godefroy2016}.

\begin{Thm}\label{NA} Let $X$ and $Y$ be two Banach spaces and $M>1$. Assume that $f:X\to Y$ and $g:Y\to X$ are continuous with $Lip_\infty(f)=1$, $Lip_\infty(g)< M$ and that there exists a constant $C\ge 0$ such that
$$\forall x\in X\ \ \|(g\circ f)(x)-x\|\le C\ \ {\rm and} \ \ \forall y\in Y\ \ \|(f\circ g)(y)-y\|\le C.$$
Assume also that $f$ attains its norm $\|f\|_{CL}=1$ in the direction $y\in S_Y$. Then
$$\forall t\in (0,1]\ \ \    \overline{\rho}_{Y}\big(y,\frac{t}{576M^3}\big)\le \overline{\rho}_X(t).$$
\end{Thm}

\begin{proof} Let us fix $\eps$ in $(0,1)$ and denote $|\ |$ the norm constructed in Theorem \ref{CLequivalence}.\\
There exists sequences $(x_n)_{n=1}^\infty$, $(x_n')_{n=1}^\infty$ in $X$ such that
$$\lim_{n\to \infty}\|x_n-x_n'\|= \infty\ \ \ {\rm and}\ \ \  \lim_{n\to \infty}\frac{f(x_n)-f(x_n')}{\|x_n-x_n'\|} = y.$$
Note that for any $k$ in $\Ndb$, $\frac{f(x_n)-f(x_n')}{\|x_n-x_n'\|}\in C_k$ for $n$ large enough. So $y\in C_k$ and $|y|_k\le 1$. Therefore,  $|y|\le 1$. On the other hand $|y|\ge (1+\eps)^{-1}\ge 1-\eps$.\\
Denote $u=\frac{y}{|y|}$. It follows from Corollary \ref{aus} that $\overline{\rho}_{|\ |}\big(u,\frac{t}{576M^2}\big)\le \overline{\rho}_X(t)+\eps.$\\
Then there exists a finite codimensional subspace $E$ of $Y$ such that for all $v\in E$ with $|v|\le \frac{t}{576M^2}$, we have
$$|u+v|\le 1+ \overline{\rho}_X(t)+2\eps.$$
It follows that for all $v\in E$ with $\|v\|\le \frac{t}{576M^3}$,
$$\|y+v\|\le (1+\eps)(|y-u|+|u+v|)\le (1+\eps)(1+\overline{\rho}_X(t)+3\eps).$$
Since $\eps>0$ is arbitrary in the above inequality, this concludes our proof.
\end{proof}

Let us now recall that a Banach space $X$ has the {\it Kadets-Klee property} if the norm and weak topologies coincide on the unit sphere of $X$. The following corollaries are the coarse Lipschitz analogues of Corollary 3.5 in \cite{Godefroy2016}.

\begin{Cor}\label{notNA} Let $X$ and $Y$ be two infinite dimensional Banach spaces such that $X$ is asymptotically uniformly flat and $Y$ has the Kadets-Klee property and assume that $f:X\to Y$ is a coarse Lipschitz equivalence. Then $f$ does not attain its norm in any direction in $S_Y$.
\end{Cor}

\begin{proof} Assume on the contrary that $f:X\to Y$ is a coarse Lipschitz equivalence and that $f$ attains its norm in the direction $y\in S_Y$. Assume also, as we may by Corollary \ref{homogeneity}, that $\|f\|_{CL}=1$. Since $X$ is asymptotically uniformly flat, it follows from Theorem \ref{NA} that there exits $t_0>0$ so that for all $t\in [0,t_0]$, $\overline{\rho}_Y(y,t)=0$. Consider now a weakly null net $(y_\alpha)_{\alpha \in A}$ in $Y$ such that $\|y_\alpha\|= t_0$ for all $\alpha$ in $A$. Then $(\|y+y_\alpha\|)_{\alpha \in A}$ tends to 1, which  contradicts the assumption that $Y$ has the Kadets-Klee property. \end{proof}

\begin{Cor} There exists a pair of Banach spaces $(X,Y)$ such that the norm attaining coarse Lipschitz maps are not dense in $CL(X,Y)$.
\end{Cor}

\begin{proof} Consider $X=(c_0,\|\ \|_\infty)$ and $Y=(c_0,\|\ \|_Y)$, where $\|\ \|_Y$ is an equivalent norm on $c_0$ with the Kadets-Klee property. We recall that such an equivalent norm exits on any separable Banach space (see for instance the book \cite{DevilleGodefroyZizler1993} and references therein). Moreover $X$ is clearly asymptotically uniformly flat. Since $X$ is an absolute retract (see Example 1.5 of Chapter 1 in \cite{BenyaminiLindenstrauss2000}), we have in particular that $(X,Y)$ and $(Y,X)$ have the net extension property. Therefore, by Propositions \ref{lipschitzrep} and \ref{open}, $\cal C \cal L \cal E(X,Y)$ can be viewed as an open subset of $CL(X,Y)$. Since it contains the identity map on $c_0$, it is a non empty open subset of $CL(X,Y)$. Combining this with Corollary \ref{notNA} finishes our proof.
\end{proof}

\noindent {\bf Aknowledgements.} The authors wish to thank G. Godefroy, M. Martin, A. Proch\'{a}zka and A. Valette for fruitful discussions on the subject of this paper.

\end{document}